\documentclass[11pt,a4paper]{amsart}
\usepackage[utf8]{inputenc}

\usepackage{amsmath}
\usepackage{amsthm}
\usepackage{amsfonts}
\usepackage{mathtools}
\usepackage[shortlabels]{enumitem}
\usepackage{booktabs}
\usepackage{url}
\usepackage{mdwlist}

\usepackage[dvipsnames]{xcolor}
\usepackage{pgf,tikz,pgfplots}
\usetikzlibrary{calc,matrix}
\usepackage{graphicx}
\usepackage{caption}
\usepackage{subcaption}

\usepackage{fancyvrb}

\usepackage{multicol}

\usepackage[a4paper,left=2.5cm,top=3.5cm,right=2.5cm,bottom=3.5cm]{geometry}

\newtheorem{thm}{Theorem}[section]
\newtheorem{lemma}[thm]{Lemma}
\newtheorem{cor}[thm]{Corollary}
\newtheorem{prop}[thm]{Proposition}

\theoremstyle{definition}
\newtheorem{defi}[thm]{Definition}
\newtheorem{ex}[thm]{Example}

\theoremstyle{remark}
\newtheorem{rem}[thm]{Remark}

\numberwithin{equation}{section}
\numberwithin{figure}{section}

\newcommand{\N}{\mathbb{N}}
\newcommand{\Z}{\mathbb{Z}}
\newcommand{\C}{\mathcal{C}}
\newcommand{\Pn}[1]{\mathbb{P}^{#1}}
\newcommand{\bfa}{\mathbf{a}}
\newcommand{\bfA}{\mathbf{A}}
\newcommand{\sg}{\mathcal{S}}
\newcommand{\ap}{\operatorname{Ap}}
\newcommand{\AP}{\operatorname{AP}}
\newcommand{\aps}{\operatorname{AP}_{\mathcal{S}}}
\newcommand{\regn}{\operatorname{reg}}
\newcommand{\bfs}{\mathbf{s}}

\newcommand{\hf}[1]{\operatorname{HF}_{#1}}
\newcommand{\hp}[1]{\operatorname{HP}_{#1}}
\newcommand{\reg}[1]{\operatorname{reg}(#1)}
\newcommand{\rhp}[1]{\operatorname{r}(#1)}
\newcommand{\pd}[1]{\operatorname{pd}(#1)}
\newcommand{\depth}[1]{\operatorname{depth}(#1)}
\newcommand{\red}[1]{\operatorname{red}(#1)}

\newcommand{\Hom}[1]{H_\mathfrak{m}^{#1} \left( k[\sg] \right)}
\newcommand{\en}[1]{\operatorname{end}\left(#1\right)}

\newcommand{\mE}{m \left( E_\sg\right)}
\newcommand{\mAP}{m \left( \operatorname{AP}_\sg \right)}

\pgfplotsset{compat=1.17}

\title{Castelnuovo-Mumford regularity of projective monomial curves via sumsets}
\author{Philippe Gimenez and Mario González-Sánchez}
\curraddr{
\texttt{Philippe Gimenez and Mario González-Sánchez:} IMUVA-Mathematics Research Institute, Universidad de Valladolid, 47011 Valladolid (Spain).
}
\email{pgimenez@uva.es;  mario.gonzalez.sanchez@uva.es}
\date{}
\thanks{This work was supported in part by the grants PID2019-104844GB-I00  
and TED2021-130358B-I00 funded by MCIN/AEI/10.13039/501100011033 and by the European Union NextGenerationEU/PRTR. The second author thanks financial support from European Social Fund, {\it Programa Operativo de Castilla y León}, and {\it Consejería de Educación de la Junta de Castilla y León}.
}

\subjclass[2020]{Primary 13D02; Secondary 13D45, 11B13, 14H45, 20M50}

\keywords{projective monomial curve, semigroup ring, Castelnuovo-Mumford regularity, sumsets, Apery set}

\begin{document}

\maketitle

\begin{abstract}
Let $A=\{a_0,\ldots,a_{n-1}\}$ be a finite set of $n\geq 4$ non-negative relatively prime integers such that $0=a_0<a_1<\cdots<a_{n-1}=d$. The $s$-fold sumset of $A$ is the set $sA$ of integers that contains all the sums of $s$ elements in $A$. On the other hand, given an infinite field $k$, one can associate to $A$ the projective monomial curve $\C_A$ parametrized by $A$,
\[
\C_A=\{(v^d:u^{a_1}v^{d-a_1}:\cdots:u^{a_{n-2}}v^{d-a_{n-2}}:u^d) \mid (u:v)\in\Pn{1}_k\}\subset\Pn{n-1}_k\,.
\]
The exponents in the previous parametrization of $\C_A$ define a homogeneous semigroup $\sg\subset\N^2$.
We provide several results relating the Castelnuovo-Mumford regularity of $\C_A$ to the behaviour of the sumsets of $A$ and to the combinatorics of the semigroup $\sg$ that reveal a new interplay between commutative algebra and additive number theory.
\end{abstract}

\section*{Introduction}
Let $A = \{a_0,a_1,\ldots,a_{n-1}\} \subset \N$ be a set of non-negative integers where we assume that $a_0<\dots<a_{n-1}$ and set $d:=a_{n-1}$. 
For every $s\in \N$, the \textit{$s$-fold sumset} of $A$, $sA$, is defined by 
$0A := \{0\}$ and for $s\geq 1$, \[sA := \{a_{i_1}+\dots+a_{i_s}: 0\leq i_1\leq \dots\leq i_s \leq n-1\}.\] 
Additive number theory studies the sumsets of $A$. 
As we will see later in (\ref{eq:sumset_HF}), for our purpose we will need to count the number of elements in $sA$.
As observed in \cite[(1.1) p.2]{nathanson_ANT}, in order to compute $|sA|$, one may assume without loss of generality that $a_0=0$ and $\gcd(a_1,\ldots,a_{n-1}) = 1$. When this occurs, $A$ is said to be in \textit{normal form}.
\newline

Consider now the points $\bfa_0 = (0,d), \ \bfa_1=(a_1,d-a_1), \ldots, \bfa_{n-1} = (d,0)$ in $\N^2$, the set $\bfA = \{\bfa_0,\bfa_1,\ \ldots,\ \bfa_{n-1}\}$, and the subsemigroup $\sg$ of $\N^2$ generated by $\bfA$. Given an arbitrary infinite field $k$, one can associate to $A$ the projective monomial curve $\C_A$ parametrized by $\bfA$: 
\[\C_{A} = \{(v^d:u^{a_1}v^{d-a_1}:\dots:u^{a_{n-2}}v^{d-a_{n-2}}:u^{d}) \mid (u:v) \in \Pn{1}_k\} \subset \Pn{n-1}_k.\] 
If $A$ is in normal form, it is an algebraic curve of degree $d$ and its defining ideal $I(\C_{A})$ is the kernel of the homomorphism of $k$-algebras $\varphi:k[x_0,\ldots,x_{n-1}] \rightarrow k[u,v]$ induced by $\varphi(x_i) = u^{a_i}v^{d-a_i}$. The ideal $I(\C_{A})$ is homogeneous, binomial and prime, i.e., it is a homogeneous toric ideal. Denoting by $k[\C_A]:=k[x_0,\ldots,x_{n-1}]/I(\C_A)$ the homogeneous coordinate ring of $\C_A$, one has that $k[\C_A]$ is isomorphic to $\hbox{Im}{\varphi}=k[\sg]$, the semigroup ring of $\sg$.
If $\hf{A}$ denotes the Hilbert function of $k[\C_A]$, by \cite[Prop. 2.3]{Elias2022} one has that
\begin{equation}\label{eq:sumset_HF}
  |sA| = \hf{A}(s), \text{ for all } s\in \N.  
\end{equation}
This provides a bridge between additive number theory and the geometry of monomial projective curves that has been recently explored in \cite{Elias2022} and later generalized to higher dimension varieties in \cite{Colarte2023}. We will follow here the same philosophy: our aim is to study some homological invariants of the projective monomial curve $\C_{A}$ that we will now define, through the sumsets of $A$, and vice versa.
\newline

Given a \textit{minimal graded free resolution} (m.g.f.r.) of the graded $k[x_0,\ldots,x_{n-1}]$-module $k[\C_A]$, 
\[\mathcal{F}: 0 \rightarrow F_p \rightarrow \ldots F_0 \rightarrow k[\C_A] \rightarrow 0 \, ,\] 
where the $F_i$'s are free modules, one has that for all $i=0,\ldots,p$, $F_i$ is generated by $\beta_{i,j}$ elements of degree $j$. The non-zero integers $\beta_{i,j}$ are invariants of the module $k[\C_A]$ called its \textit{graded Betti numbers}, and we can arrange them in the \textit{Betti dia\-gram} of $k[\C_A]$, a table whose entry in column $i$ and row $j$ is $\beta_{i,i+j}$. The size of the Betti diagram of $k[\C_A]$ is measured by two important invariants of $k[\C_A]$: the label of the last column is the \textit{projective dimension} of $k[\C_A]$, $\pd{k[\C_A]} = p$, the index of the last free module in any m.g.f.r. of $k[\C_A]$, while the label of the last row is its \textit{Castelnuovo-Mumford regularity}, 
\begin{equation}\label{eq:defReg}
\reg{k[\C_A]} := \max \{j-i: \beta_{i,j} \neq 0,\ 0\leq i\leq p,\ j\geq 0\} \, .
\end{equation}
The projective dimension is controlled by the  Auslander-Buchsbaum formula: $\pd{k[\C_A]} = n-\depth{k[\C_A]}$. As the Krull dimension of $k[\C_A]$ is $2$ and the ideal $I(\C_A)$ is prime, the depth of $k[\C_A]$ can only be $1$ or $2$. Thus, the projective dimension is either $n-2$ if $k[\C_A]$ is Cohen-Macaulay, or $n-1$ otherwise. 
The behaviour of the Castelnuovo-Mumford regularity of $k[\C_A]$ is more chaotic. 
In this paper, the sumsets of $A$ and the combinatorics of the semigroup $\sg$ will be related to the Castelnuovo-Mumford regularity and the regularity of the Hilbert function of $k[\C_A]$, revealing a nice interplay between additive number theory and commutative algebra.
Note that if $n=2$, $A = \{0,1\}$ and if $n=3$, $\C_A$ is a hypersurface, so we will assume here that $n\geq 4$.
\newline

The paper is structured as follows. In section \ref{sec:background}, we recall some results in additive number theory, in particular the fundamental Structure Theorem and its relation to monomial curves. We define the sumsets regularity $\sigma(A)$ of a finite set of integers in normal form $A$ as the least integer such that, for all larger integers, the decomposition in the Structure Theorem holds. Several upper bounds for $\sigma(A)$ that appear in the literature are recalled, in particular the Granville-Walker bound recently obtained in \cite{Granville2021}. In section \ref{sec:semigroup}, we analyze the structure of the semigroup $\sg$ and see that the sumsets regularity of $A$ defined in the previous section could also be called the conductor of the semigroup $\sg$. We focus on two important finite subsets of the semigroup $\sg$ that will play in fundamental role later: its Apery set and its exceptional set. Both subsets can be used to characterize the Cohen-Macaulay property for $k[\C_A]$ as shown in Proposition \ref{prop:charcM}. Section \ref{sec:mainresults} contains our main results. We start by completing the characterization of the elements in the Structure Theorem given in \cite[Prop. 3.4]{Elias2022} and express the sumsets re\-gu\-la\-ri\-ty of $A$ in terms of some invariants of the monomial curve $\C_A$ in Theorem \ref{thm:sigma}. 
As a direct consequence, we give a new upper bound for the sumsets regularity in Theorem \ref{thm:boundSigma}.
We also give a combinatorial way for computing the Castelnuovo-Mumford regularity of $k[\C_A]$ in terms of the Apery and the exceptional sets of $\sg$ (Theorem \ref{thm:reg}) and provide both upper and lower bounds for the Castelnuovo-Mumford regularity of $k[\C_A]$ in terms of the conductor of $\sg$ (Theorem \ref{thm:bounds_reg}). In section \ref{sec:shape}, we prove in Theorem \ref{thm:shapeGeneral} a general result that allows to read on the Betti diagram the value of the difference between the Castelnuovo-Mumford regularity and the regularity of the Hilbert function of $k[\C_A]$. Applied to the monomial curve $\C_A$, we deduce in Theorem \ref{thm:last_step} a way to characterize when the regularity is attained at the last step of a m.g.f.r. Finally, in section \ref{sec:consequences} we use our results to relate a recent result in additive number theory, the Granville-Walker bound for the sumsets regularity, to a classical result in algebraic geometry, the Gruson-Lazarsfeld-Peskine bound for the Castelnuovo-Mumford regularity in the particular case of monomial curves. More precisely, we show how to obtain the first bound from the second and vice versa. 
\newline

The computations in the examples given in this paper are performed by using Singular \cite{Singular} and, in particular, the library {\tt mregular.lib} \cite{mregular.lib}. We also used the package {\tt NumericalSgps} \cite{NumericalSgps} of GAP.

\subsection*{Notations}
In this paper, $\N = \{0,1,2,\ldots\}$. For any $a,b\in \Z$ such that $a\leq b$, we denote $[a,b] := \{n\in \Z: a\leq n \leq b\}$. If $x\in \Z$, $\lfloor x \rfloor$ is the greatest integer less than or equal to $x$ (floor function), while $\lceil x \rceil$ is the least integer greater than or equal to $x$ (ceil function). If $d\in \N$ and $A\subset \N$, we denote $d-A := \{d-a: a\in A\}$. Furthermore, we will assume that all the semigroups have an identity, i.e., we don't distinguish between semigroup and monoid. 

If $R = \oplus_{s\in \N}R_s$ is a standard graded $k$-algebra, we denote by $\hf{R}$ and $\hp{R}$ its Hilbert function and Hilbert polynomial respectively. 
The least integer $r$ such that, for all integer $s\geq r$, $\hf{R}(s)=\hp{R}(s)$ is called the \textit{regularity of the Hilbert function} of $R$ and we will denote it by $\rhp{R}$.
The Castelnuovo-Mumford regularity of $R$ will be denoted by $\reg{R}$ and we will use the abbreviation \textit{m.g.f.r.} for minimal graded free resolution.

Finally, when we draw part of a semigroup $\sg \subset \N^2$ as in Figures \ref{fig:struct_tildeA} and \ref{fig:apsE}, filled circles represent points in $\sg$ while empty squares represent points outside $\sg$, i.e., gaps of $\sg$.

\section{The Structure Theorem} \label{sec:background}
In this section we give an overview of some results in additive number theory and their connection to monomial curves. 
Let's first recall the so-called Structure Theorem, one of the main results in additive number theory.

\begin{thm}[Structure Theorem, {\cite[Thm. 1.1]{nathanson_ANT}}] \label{thm:structure}
    If $A=\{a_0=0<a_1<\dots<a_{n-1}=d\}\subset \N$ is a finite set in normal form, then there exist integers $c_1,c_2 \in \mathbb{N}$ and finite subsets $C_i \subset [0,c_i-2]$, $i=1,2$, such that
    \begin{equation}
        sA = C_1 \sqcup [c_1,sd-c_2] \sqcup \left( sd - C_2\right)
        \label{eq:struct_thm}
    \end{equation} 
    for all $s\geq \max\{1,s_0\}$ where $s_0:= (n-2)(d-1)d$.
\end{thm}

The elements in the Structure Theorem have recently been characterized in \cite[Prop. 3.4]{Elias2022} in terms of the curve $\C_A$ and some of its invariants.
If $A$ is a finite set in normal form, it is well known that $\C_A$ has two possible singular points, $P_1 = (1:0:\dots:0) \in \Pn{n-1}_k$ and $P_2 = (0:\dots:0:1) \in \Pn{n-1}_k$. Moreover, if $\delta(\C_A,P)$ denotes the singularity order of $P$, then $\delta(\C_A,P_1) = |\N \setminus \sg_1|$ and $\delta(\C_A,P_2) = |\N \setminus \sg_2|$, where $\sg_1$ and $\sg_2$ denote the numerical semigroups generated by $A$ and $d-A$, respectively.
Using that $\C_A$ has degree $d$, one gets ({\cite[Prop. 3.1]{Elias2022}}) that for all $s\geq \rhp{k[\C_A]}$, 
    \begin{equation}
     |sA| = \hf{A}(s) = sd+1-\delta(\C_A,P_1)-\delta(\C_A,P_2) \, .
     \label{eq:Elias_3.1}
    \end{equation}

\begin{prop}[{\cite[Prop. 3.4]{Elias2022}}] \label{prop:Elias_3.4}
Following notations in Theorem \ref{thm:structure}, for $i=1,2$ the following claims hold:
\begin{enumerate}[(1)]
\item $c_i$ is the conductor of $\sg_i$.
\item\label{number_Elias_2} $C_i = \sg_i \cap [0,c_i-2]$.
\item $\delta(\C_A,P_i) = c_i-|C_i|$.
\end{enumerate}
\end{prop}

\begin{defi}
    The least integer $\sigma$ such that the decomposition \eqref{eq:struct_thm} in Theorem \ref{thm:structure} holds for all $s\geq\sigma$ will be called the {\it sumsets regularity} of $A$ and we will denote it by $\sigma(A)$.
\end{defi}

Theorem \ref{thm:structure} provides an upper bound for $\sigma(A)$ that is generally far from its real value:  $\sigma(A)\leq (n-2)(d-1)d$. After Nathanson's proof, other proofs of Theorem \ref{thm:structure} have been published, \cite{Wu2011,Granville2020,Granville2021}. In these articles, the authors give the following better upper bounds for $\sigma(A)$: 
    \begin{itemize}
        \item \cite[Thm. 2]{Wu2011} (Wu, Chen, Chen; 2011) $\sigma(A) \leq  \left( \sum_{i=2}^{n-2} a_i \right) +d-n+1=: s_0^{WCC}$.
        \item \cite[Thm. 1]{Granville2020} (Granville, Shakan; 2020) $\sigma(A) \leq  2 \lfloor \frac{d}{2} \rfloor=: s_0^{GS}$.
        \item \cite[Thm. 1]{Granville2021} (Granville, Walker; 2021) $\sigma(A) \leq  
        d-n+2=: s_0^{GW}$.
\end{itemize}

Note that in \cite{Wu2011,Granville2020,Granville2021}, the union in equation \eqref{eq:struct_thm} is not shown to be disjoint but this is shown in \cite{Lev2022} for the Granville-Walker bound and, as $s_0^{WCC}>s_0^{GW}$ and $s_0^{GS}>s_0^{GW}$ if $n\geq 4$, the above claims hold. \newline

Besides giving a great upper bound for $\sigma(A)$, Granville and Walker also characterize the sets $A$ for which this bound is attained.

\begin{thm}[{\cite[Thm. 2]{Granville2021}}] \label{thm:Granville_Walker}
    Let $n\in \N$, $n\geq 3$, and $A = \{a_0=0<a_1<\dots<a_{n-1}=d\}\subset \N$ be a set in normal form. Then, $\sigma(A) = d-n+2$ if, and only if, either $A$ or $d-A$ belongs to one of the following two families:
    \begin{itemize}
        \item $A = [0,d] \setminus \{a\}$, for some $a$ such that $2\leq a\leq d-2$;
        \item $A = [0,1] \sqcup [a+1,d]$, for some $a$ such that $2\leq a \leq d-2$.
    \end{itemize}
\end{thm}
Note for any $A$ belonging to one of the two families in Theorem \ref{thm:Granville_Walker}, the monomial curve $\C_A$ is smooth.

\section{The structure of the homogeneous semigroup and its relation to the sumsets} \label{sec:semigroup}
As already observed, associated to a set of integers $A = \{a_0=0<a_1<\dots<a_{n-1}=d\}$ in normal form, one has the set \[\bfA = \{\bfa_0,\bfa_1,\ldots,\bfa_{n-1}\} \subset \N^2 \, ,\] where $\bfa_i=(a_i,d-a_i)$ for all $i=0,\ldots,n-1$, that we will call its \textit{homogenization}. A semigroup $\sg$ in $\N^2$ generated by a set $\bfA$ of this form will be said to be \textit{homogeneous of degree $d$}. \newline

It is trivial to verify that the sumsets of $\bfA$ are completely determined by those of $A$ since, for each $s\in \N$, \[s\bfA = \{(\alpha,sd-\alpha): \alpha\in sA\} \, .\] 
In particular, for any $s\in\N$, $|sA|=|s\bfA|$. 
Furthermore, the semigroup $\sg$ generated by $\bfA$ satisfies that $\sg  = \sqcup_{s=0}^\infty s \bfA$. Note that each $s\bfA$ lies on the ``line" $L_s:=\{(x,y)\in \N^2: x+y=sd\}$.
\newline

We can apply the Structure Theorem to improve our knowledge on the sumsets of $\bfA$ and the semigroup $\sg$. By Theorem \ref{thm:structure} and Proposition \ref{prop:Elias_3.4}, we have that for all $s\geq \sigma(A)$, $s\bfA$ consists on a central interval and, outside that interval, a copy of the non-trivial part of the semigroups $\sg_1$ and $\sg_2$, i.e., for all $s\geq \sigma(A)$, 
\[
s\bfA=\{(i,sd-i),\,i\in\sg_1 \cap [0,c_1-2]\} \sqcup \{(i,sd-i),\,i\in[c_1,sd-c_2]\} \sqcup \{(sd-i,i),\,i\in\sg_2 \cap [0,c_2-2]\}.
\]
Furthermore, $\sigma(A)$ is the least integer such that this decomposition is satisfied for all $s\geq\sigma(A)$. More precisely, for $s\geq\sigma(A)$, when we go from $s\bfA$ to $(s+1)\bfA$, gaps coming from $\sg_1$ move up while gaps coming from $\sg_2$ move to the right, and there are no other gaps in $(s+1)\bfA$ than the ones coming from $s\bfA$, as shown in Figure \ref{fig:struct_tildeA}. And $\sigma(A)$ is the least integer such that this occurs. For this reason, the regularity of the sumsets of $A$, $\sigma(A)$, could also be called the \textit{conductor} of the homogeneous semigroup $\sg$ and  denoted by $\sigma(\sg)$. If no confusion arises, from now on we will simply denote this number by $\sigma$, i.e., $\sigma=\sigma(\sg)=\sigma(A)$. \newline

\begin{figure}[htbp]
    \centering
    \begin{tikzpicture}[scale=0.4]%[scale=0.4,every node/.style={scale=0.4}]
  % Level s
  \draw[thin,black]  (2,16) -- (14,4);
  \draw[thin,black,dashed]  (0.5,17.5) -- (2,16);
  \draw[thin,black,dashed]  (14,4) -- (15.5,2.5);
  \draw [red] plot [only marks, mark size=3.5, mark=square*,mark options = {fill=white}] coordinates {(3,15) (5,13) (6,12) (11,7) (13,5)};
  \draw [blue] plot [only marks, mark size=3.5, mark=*] coordinates {(4,14) (7,11) (8,10) (9,9) (10,8) (12,6)};
  % Level s+1
  \draw[thin,black] (2,20) -- (18,4);
  \draw[thin,black,dashed]  (0.5,21.5) -- (2,20);
  \draw[thin,black,dashed]  (18,4) -- (19.5,2.5);
  \draw [red] plot [only marks, mark size=3.5, mark=square*,mark options = {fill=white}] coordinates {(3,19) (5,17) (6,16) (15,7) (17,5)};
  \draw [blue] plot [only marks, mark size=3.5, mark=*] coordinates {(4,18) (7,15) (8,14) (9,13) (10,12) (11,11) (12,10) (13,9) (14,8) (16,6)};
  % Arrows
  \draw [-stealth,red](3,15.5) -- (3,18.5);
  \draw [-stealth,red](5,13.5) -- (5,16.5);
  \draw [-stealth,red](6,12.5) -- (6,15.5);
  \draw [-stealth,red](11.5,7) -- (14.5,7);
  \draw [-stealth,red](13.5,5) -- (16.5,5);
  % Labels
  \draw[thin,black] (5.75,9.75) -- (5.5,9.5) -- (8.5,6.5) -- (8.75,6.75);
  \draw[thin,black] (4.75,10.75) -- (4.5,10.5) -- (2,13);
  \draw[thin,black,dashed] (2,13) -- (0.5,14.5);
  \draw[thin,black] (9.75,5.75) -- (9.5,5.5) -- (11,4);
  \draw[thin,black,dashed] (11,4) -- (12.5,2.5);
  \node[align=left] at (4.75,7) {Central\\ interval};
  \node[align=left] at (-0.5,11.5) {Non-trivial\\ part of $\sg_1$};
  \node[align=left] at (7.5,3.5) {Non-trivial \\ part of $\sg_2$};
  \node[align=left] at (15.5,1.75) {$s\bfA$};
  \node[align=left] at (19.5,1.75) {$(s+1)\bfA$};
\end{tikzpicture}
    \caption{Structure of the sumsets of $\bfA$. For $s\geq \sigma$, we distinguish three disjoint areas: the central interval and the copies of the non-trivial parts of $\sg_1$ and $\sg_2$.}
    \label{fig:struct_tildeA}
\end{figure}

We can relate the conductor of the semigroup $\sg$ to the Hilbert function regularity of $k[\C_A]$ on the one hand, and to the conductors of the semigroups $\sg_1$ and $\sg_2$ on the other. This relation will become more precise later in Theorem \ref{thm:sigma}.
\begin{lemma}\label{lemma:spoiler}
    Let $A=\{a_0=0<a_1<\dots<a_{n-1}=d\}\subset \N$ be a finite set in normal form and $\sigma$ be its sumsets regularity. 
    \begin{enumerate}[(1)]
        \item\label{lemma:spoiler_1} For all $s\geq\sigma$, $|L_s\setminus s\bfA|=|L_\sigma\setminus \sigma\bfA|$ and $|(s+1)A|-|sA|=d$. In particular, $\sigma\geq \rhp{ k[\C_A]}$.
        \item\label{lemma:spoiler_2} If $sd-c_2<c_1$, the central interval in the previous decomposition of $s\bfA$ does not exist, and hence $\sigma\geq \lceil \frac{c_1+c_2}{d} \rceil$.
    \end{enumerate}
\end{lemma}

\begin{proof}
Both results are consequences of the discussion before Figure \ref{fig:struct_tildeA}. \ref{lemma:spoiler_2} is direct and for \ref{lemma:spoiler_1}, recall that for all $s\geq 0$, $\hf{A}(s)= |sA|$ and if $s\geq\sigma$, 
\begin{equation}
\begin{split}
\hf{A}(s) &= |sA| = sd+1-\left( c_1-|C_1|+c_2-|C_2|\right)\\
&= sd+1-\delta(\C_A,P_1)-\delta(\C_A,P_2) = \hp{A}(s)
\end{split}
\label{eq:dem_prop_3.4}
\end{equation}
by Proposition \ref{prop:Elias_3.4}, so $\sigma \geq \rhp{k[\C_A]}$. 
\end{proof}

\begin{rem}
    Note that both inequalities for $\sigma$ in Lemma \ref{lemma:spoiler} can be strict as we will see later in Example \ref{ex:sigma}.
\end{rem}

Let's now focus on the three semigroups $\sg_1$, $\sg_2$ and $\sg$. For $i=1,2$, we define the \textit{Apery set of $\sg_i$ with respect to $d$} as $\ap_i := \{a\in \sg_i: a-d\notin \sg_i\}$. We know that $\ap_i$ is a complete set of residues modulo $d$, and hence $\ap_1 = \{r_0=0,r_1,\ldots,r_{d-1}\}$ and $\ap_2 = \{t_0=0,t_1,\ldots,t_{d-1}\}$ with $r_i\equiv t_i \equiv i \pmod{d}$. 

\begin{defi} \label{defi:apsE}
The \textit{Apery set} $\aps$ of $\sg$ and the \textit{exceptional set} $E_\sg$
of $\sg$ are defined as follows:
\begin{itemize}
    \item $\aps := \{(x,y)\in \sg: (x,y) - \bfa_0 \notin \sg, (x,y)-\bfa_{n-1} \notin \sg\}$.
    \item $E_\sg := \{(x,y)\in \sg: (x,y) - \bfa_0 \in \sg, (x,y)-\bfa_{n-1} \in \sg, (x,y)- \bfa_0-\bfa_{n-1} \notin \sg\}$.
\end{itemize}  
Moreover, for each $s\in \N$, set $\operatorname{AP}_s := \aps \cap L_s = \aps \cap s\bfA$ and $E_s := E_\sg \cap L_s= E_\sg \cap s\bfA$.
Figure \ref{fig:apsE} shows how points in $E_\sg$ and $\aps$ look like when one draws the semigroup $\sg$.
\end{defi}

\begin{figure}[htbp]
\centering
\begin{subfigure}[b]{0.45\linewidth}
\centering
\begin{tikzpicture}[scale=1.5]
  \draw[thin,black]  (-0.5,1.5) -- (1.5,-0.5);
  \draw[thin,black] (-0.5,0.5) -- (0.5,-0.5);
  \draw[thin,black] (-0.5,0.5) -- (0.5,-0.5);
  \draw[thin,black] (0.5,1.5) -- (1.5,0.5);
  \draw[thin,black,dashed]  (0,0) -- (0,1);
  \draw[thin,black,dashed]  (0,0) -- (1,0);
  \draw[thin,black,dashed]  (1,0) -- (1,1);
  \draw[thin,black,dashed]  (0,1) -- (1,1);
  \draw [red] plot [only marks, mark size=1, mark=square*,mark options = {fill=white}] coordinates {(0,0)};
  \draw [blue] plot [only marks, mark size=1, mark=*] coordinates {(1,0) (0,1) (1,1)};
  \node[align=left] at (1.4,1.15) {$(x,y)$};
  \node[align=left] at (0.5,-0.75) {$L_{s-2}$};
  \node[align=left] at (1.5,-0.75) {$L_{s-1}$};
  \node[align=left] at (1.75,0.4) {$L_s$};
\end{tikzpicture}
%\caption{}
\label{fig:m1(A)}
\end{subfigure}
\begin{subfigure}[b]{0.45\linewidth}
\centering
\begin{tikzpicture}[scale=1.5]
  \draw[thin,black]  (-0.5,1.5) -- (1.5,-0.5);
  \draw[thin,black]  (0.5,1.5) -- (1.5,0.5);
  \draw[thin,black] (-0.5,0.5) -- (0.5,-0.5);
  \draw[thin,black,dashed]  (0,0) -- (0,1);
  \draw[thin,black,dashed]  (0,0) -- (1,0);
  \draw[thin,black,dashed]  (1,1) -- (0,1);
  \draw[thin,black,dashed]  (1,1) -- (1,0);
  \draw [blue] plot [only marks, mark size=1, mark=*] coordinates {(1,1)};
  \draw [red] plot [only marks, mark size=1, mark=square*,mark options = {fill=white}] coordinates {(1,0) (0,1) (0,0)};
  \node[align=left] at (1.4,1.15) {$(x',y')$};
  \node[align=left] at (0.5,-0.75) {$L_{s-2}$};
  \node[align=left] at (1.5,-0.75) {$L_{s-1}$};
  \node[align=left] at (1.75,0.4) {$L_s$};
\end{tikzpicture}
%\caption{} 
\label{fig:m2(A)}
\end{subfigure}
\caption{A point $(x,y)$ in $E_s$ and a point $(x',y')$ in $\operatorname{AP}_s$.}
\label{fig:apsE}
\end{figure}

\begin{rem}\label{rem:bounds_apsE}
As a consequence of Theorem \ref{thm:structure}, one gets that, if $\sigma$ is the conductor of $\sg$, then \[\forall s\geq \sigma + 2,\ \operatorname{AP}_s = E_s = \emptyset\,.\]
\end{rem}

The Cohen-Macaulayness of $k[\C_A]$ is characterized in terms of $\aps$ and $E_\sg$ as we will show in Proposition \ref{prop:charcM} . Let's previously prove the following easy lemma. 

\begin{lemma} \label{lemma:CMchar}
For all $i=1,\ldots,d-1$, the following claims hold:
    \begin{enumerate}[(1)]
        \item\label{lemma:CMchar_1} 
        If $(r_i,t_{d-i}) \in \sg$, then $(r_i,t_{d-i}) \in \aps$.
        \item\label{lemma:CMchar_2} If $(r_i,t_{d-i}) \notin \sg$, then $(r_i,t_{d-i}) \notin \aps$
        and there exist natural numbers $x>r_i$ and $y>t_{d-i}$ such that $(x,t_{d-i}) \in \aps$ and $(r_i,y) \in \aps$.
    \end{enumerate}
\end{lemma}

\begin{proof}
\ref{lemma:CMchar_1} is trivial. In order to prove \ref{lemma:CMchar_2}, take $i \in \{1,2,\ldots,d-1\}$. Since $r_i \in \mathcal{S}_1$, there exists a natural number $y>t_{d-i}$ such that $(r_i,y) \in \mathcal{S}$ and if we choose the least $y \in \N$ satisfying this property, then $(r_i,y) \in \aps$. The proof of the existence of $x$ is analogous.
\end{proof}

Denote by $G$ the subgroup of $\Z^2$ generated by $\sg$ and set $\sg' := G \cap \left( \sg_1 \times \sg_2 \right)$. 

\begin{prop}[Characterization of the Cohen-Macaulayness of $\C_A$] \label{prop:charcM}
    The following statements are equivalent:
    \begin{enumerate}[(1)]
        \item\label{prop:charCMa} $\C_A$ is arithmetically Cohen-Macaulay, i.e., the ring $k[\C_A]$ is Cohen-Macaulay.
        \item\label{prop:charCMb} For all $i=1,\ldots,d-1$, $(r_i,t_{d-i}) \in \sg$. In other words, if $q_1\in \ap_1$, $q_2\in \ap_2$ and $q_1+q_2 \equiv 0 \pmod{d}$, then $(q_1,q_2) \in \aps$.
        \item\label{prop:charCMc} $\aps = \{(0,0)\} \cup \{(r_i,t_{d-i}):1\leq i <d\}$.
        \item\label{prop:charCMd} $\aps$ has exactly $d$ elements.
        \item\label{prop:charCMe} The exceptional set $E_\sg$ is empty.
        \item\label{prop:charCMf} $\sg'=\sg$.
    \end{enumerate}
\end{prop}

\begin{proof}
The equivalences $\ref{prop:charCMa} \Leftrightarrow \ref{prop:charCMe}$ and $\ref{prop:charCMa}\Leftrightarrow \ref{prop:charCMf}$ are well known; see, e.g., \cite[Lemma 4.3, Thm. 4.6]{Cavaliere1983}. Moreover, the implications $\ref{prop:charCMc} \Rightarrow \ref{prop:charCMb}$ and $\ref{prop:charCMc} \Rightarrow \ref{prop:charCMd}$ are trivial and $\ref{prop:charCMd} \Rightarrow \ref{prop:charCMc}$ is a direct consequence of Lemma \ref{lemma:CMchar}, so let us prove $\ref{prop:charCMb} \Leftrightarrow \ref{prop:charCMe} \Rightarrow \ref{prop:charCMc}$. \\
\underline{$\ref{prop:charCMe} \Leftrightarrow \ref{prop:charCMb}$}: Suppose that there is an index $i$, $1\leq i<d$, such that $(r_i,t_{d-i}) \notin \mathcal{S}$. By Lemma \ref{lemma:CMchar} \ref{lemma:CMchar_2}, there exist $x> r_i$ and $y> t_{d-i}$ such that $(x,t_{d-i}) \in \aps$ and $(r_i,y) \in \aps$. Then, there exist $x'\leq x$ and $y'\leq y$ such that $(x',y') \in E_\sg$, so $E_\sg$ is not empty.\\ 
Conversely, suppose that there exist $(x,y) \notin \mathcal{S}$ such that $(x+d,y) \in \mathcal{S}$ and $(x,y+d) \in \mathcal{S}$ and let $i$ be the index, $1\leq i\leq d-1$, such that $x\equiv i \equiv r_i \pmod{d}$ and $y \equiv d-i\equiv t_{d-i} \pmod{d}$. As $(x,y+d) \in \mathcal{S}$, $x\in \sg_1$, and $y\in \sg_2$ because $(x+d,y) \in \mathcal{S}$, so $r_i\leq x$ and $t_{d-i} \leq y$. This implies that $(r_i,t_{d-i}) \notin \mathcal{S}$. \\
\underline{$\ref{prop:charCMe}+\ref{prop:charCMb} \Rightarrow \ref{prop:charCMc}$}: 
Assuming that \ref{prop:charCMb} holds, one gets that $ \{(0,0)\} \cup \{(r_i,t_{d-i}):1\leq i <d\} \subset\aps$ by Lemma \ref{lemma:CMchar} \ref{lemma:CMchar_1}. In order to prove the equality, take $(x,y) \in \aps$. If $x\notin\ap_1$, then $x-d\in \mathcal{S}_1$, so there exists $y'>y$ such that $(x-d,y') \in \mathcal{S}$ and choosing $y'$ minimum with this property, one gets that $(x-d,y') \in \mathcal{S}$, $(x,y'-d) \in \mathcal{S}$ and $(x-d,y'-d) \notin \mathcal{S}$, a contradiction with \ref{prop:charCMe}. This implies that $x\in\ap_1$, and we prove that $y\in \ap_2$ using a similar argument. Thus $(x,y)=(r_i,t_{d-i})$ for some $i$, $1\leq i <d$, and we are done.
\end{proof}

\begin{rem}
{\hspace{1mm}}
\begin{enumerate}[(1)]
    \item If $k[\sg]$ is not Cohen-Macaulay, the ring $k[\sg']$ is called the \textit{Cohen-Macaulayfication} of $k[\sg]$. This is because $\sg \neq \sg'$ by Proposition \ref{prop:charcM} \ref{prop:charCMf} and $k[\sg']$ is the least Cohen-Macaulay intermediate between $k[\sg]$ and its field of fractions; see \cite[Remark 4.7]{Cavaliere1983}.
    \item For a general affine semigroup ring $\sg$, the Cohen-Macaulay property of the semigroup ring $k[\sg]$ may depend on the characteristic of the field $k$, as shown in \cite{Hoa1991}. However, by Proposition \ref{prop:charcM}, it is clear that this is not the case for a homogeneous semigroup $\sg\subset\N^2$. 
\end{enumerate}
\end{rem}

\begin{ex}\label{ex:01238}
Let $A = \{0,1,2,3,8\}\subset \N$. One can check that 
the Apery sets of $\sg_1$ and $\sg_2$ are $\ap_1 = \{0,1,2,3,4,5,6,7\}$ and $\ap_2 = \{0,17,10,11,12,5,6,7\}$, respectively, and $\aps = \{(0,0),(1,7),(2,6),(3,5),(4,12),(5,11),(6,10),(7,17)\}$, and hence $k[\C_A]$ is Cohen-Macaulay.
\end{ex}

We focus now on the distribution of points $(x,y)$ in $\aps$ and $E_\sg$ on the levels given by the sumsets of $A$. 

\begin{prop}\label{prop:levels_apsE}
    $|\operatorname{AP}_s| - |E_{s}| = |sA|-2|(s-1)A|+|(s-2)A|$, for all $s\in \N$.
\end{prop}

\begin{proof} 
Let's count the number of elements in $\AP_s$ for all $s\in\N$.
Note that $|\AP_0| = 1 = |0A|$ and $|\AP_1|= |\bfA|-2 = |A|-2|0A|$, and since $E_0=E_1=\emptyset$ and $sA=\emptyset$ if $s<0$, one gets that the formula holds if $s\leq 1$. Consider now $s\geq 2$.
Since for each element $\bfs \in (s-1)\bfA$, neither $\bfs+\bfa_0$ nor $\bfs+\bfa_{n-1}$ belong to $\AP_s$, every element in $(s-1)\bfA$ provides two elements in $s\bfA$ that do not belong to $\AP_s$ and any other element in $s\bfA$ belongs to $\AP_s$. But we are counting some of those elements twice, 
precisely the $\bfs \in s\bfA$ such that $\bfs-\bfa_0\in (s-1) \bfA$ and $\bfs-\bfa_{n-1} \in (s-1) \bfA$. Now for such an element $\bfs$, either $\bfs-\bfa_0-\bfa_{n-1}\notin (s-2) \bfA$ and hence $\bfs \in E_s$, or $(x,y)-\bfa_0-\bfa_{n-1}\in (s-2)\bfA$. This provides the following formula,
\[|\operatorname{AP}_s| = |s\bfA|-2|(s-1)\bfA|+ \left( |(s-2)\bfA| + |E_s|\right)\, ,\] 
and the result follows.
\end{proof}

\begin{rem}
As a consequence of the previous theorem and Remark \ref{rem:bounds_apsE}, we obtain that $|\aps| = |E_\sg|+d$ since \[\begin{split}
|\aps| &= \sum_{s=0}^{\sigma+1} |\AP_s| = \sum_{s=0}^{\sigma+1} \left( |sA|-2|(s-1)A|+|(s-2)A| \right) + \sum_{s=0}^{\sigma+1} |E_{s}|\\
&= \left(|(\sigma+1)A|-|\sigma A|\right) + |E_\sg| = |E_\sg|+d \, ,
\end{split}\] 
where we have that $|(\sigma+1)A|-|\sigma A|=d$ by Lemma \ref{lemma:spoiler} \ref{lemma:spoiler_1}. In particular, $|\aps|\geq d$ and we recover that $\ref{prop:charCMd}\Leftrightarrow \ref{prop:charCMe}$ in Proposition \ref{prop:charcM}.
\end{rem}

\begin{cor} \label{cor:CMseq}
If $\C_A$ is arithmetically Cohen-Macaulay, the sequence $\left( |sA|-|(s-1)A|\right)_{s=0}^{\infty} \subset \N$ is increasing (and it stabilizes at $d$).
\end{cor}

\begin{proof}
    For each $s\in \N$, we observe that \[|sA| -|(s-1)A| = \sum_{j=0}^s \left( |jA|-2|(j-1)A|+|(j-2)A| \right) = \sum_{j=0}^s |\operatorname{AP}_j| \, ,\] by Proposition \ref{prop:levels_apsE}.
\end{proof}

\begin{rem}
The result in Corollary \ref{cor:CMseq} holds in a more general setting. For a graded (or local) $k$-algebra $R$ of Krull dimension two, the differences between two consecutive elements in the sequence $(\hf{R}(s)-\hf{R}(s-1))_{s=0}^{\infty}$ are the coefficients of its $h$-polynomial that are known to be non-negative when $R$ is Cohen-Macaulay \cite{Stanley1978}. Thus, the sequence $(\hf{R}(s)-\hf{R}(s-1))_{s=0}^{\infty}$ is increasing.
\end{rem}

Note that if one removes the Cohen-Macaulay hypothesis, then the result in Corollary \ref{cor:CMseq} may be wrong as the first example below shows. But this property does not characterize arithmetically Cohen-Macaulay curves as the second example shows.

\begin{ex} \label{ex:2.13}
\begin{enumerate}[(1)]
    \item For $A = \{0,1,3,11,13\}$, $\left( |sA|-|(s-1)A|\right)_{s=0}^{\infty} = (1,4,9,14,17,15,13,$ $13,\dots)$ is not increasing, hence $k[\C_A]$ is not Cohen-Macaulay by Corollary \ref{cor:CMseq}.
    \item \cite[Ex. 4.3]{Bermejo2006sat}. For $A = \{0,5,9,11,20\}$, $\left( |sA|-|(s-1)A|\right)_{s=0}^{\infty} = (1,4,9,15,20,20,\dots)$ is increasing but $k[\C_A]$ is not Cohen-Macaulay.
\end{enumerate}
\end{ex}

\section{Regularity, sumsets and semigroups}\label{sec:mainresults}
We start this section by giving a characterization of $\sigma$, the sumsets regularity of $A$, which is also the conductor of the semigroup $\sg$, in terms of the curve $\C_A$ and its invariants. This result already appears in \cite{TFM_Mario} and it concludes the characterization of the elements in Structure Theorem given in \cite[Prop. 3.4]{Elias2022}.

\begin{thm} \label{thm:sigma}
The least integer $\sigma$ such that the decomposition \eqref{eq:struct_thm} in Theorem \ref{thm:structure} holds for all $s\geq\sigma$ is
\begin{equation*}
\sigma = \max\left\{\rhp{ k[\C_A]},\left\lceil \dfrac{c_1+c_2}{d} \right\rceil\right\}
\label{eq:sigma_regHilbert}
\end{equation*}
where $\rhp{ k[\C_A]}$ is the regularity of the Hilbert function of $k[\C_A]$ and $c_i$ is the conductor of the numerical semigroup $\sg_i$ for $i=1,2$.
\end{thm}

\begin{proof} 
By Lemma \ref{lemma:spoiler}, $\sigma\geq \max\left\{\rhp{ k[\C_A]},\lceil \frac{c_1+c_2}{d} \rceil\right\}$.
Conversely, for $s\geq \max\left\{ \rhp{ k[C_A]}, \lceil \frac{c_1+c_2}{d} \rceil \right\}$, one has that \eqref{eq:dem_prop_3.4} is satisfied by applying \eqref{eq:Elias_3.1}. Moreover, since $sd-c_2\geq c_1$, one has that \[\begin{split}
sA &= \left( sA \cap C_1\right) \sqcup  \left( sA \cap [c_1,sd-c_2]\right) \sqcup \left( sA \cap \left( sd-C_2 \right) \right) \\
&\subset C_1 \sqcup [c_1,sd-c_2] \sqcup \left( sd-C_2\right)
.
\end{split}\] Since both sets $sA$ and $C_1 \sqcup [c_1,sd-c_2] \sqcup \left( sd-C_2\right)$ are finite and have the same cardinality, they are equal, so $\max\left\{ \rhp{ k[C_A]}, \lceil \frac{c_1+c_2}{d} \rceil \right\}\geq \sigma$ and the resut follows.
\end{proof}

Given a subset $A\subset \N$ in normal form, it is not easy to know in advance whether $\sigma = \rhp{k[\C_A]}$ or $\sigma = \lceil \frac{c_1+c_2}{d} \rceil$. But in some cases it is, as Proposition \ref{prop:sigma_smoothCM} and Corollary \ref{cor:sigma_diff} show.

\begin{prop}\label{prop:sigma_smoothCM}
\begin{enumerate}[(1)]
\item\label{prop:sigma_smoothCM_1} If $\C_A$ is smooth, then $\sigma = \rhp{k[\C_A]}\geq \lceil \frac{c_1+c_2}{d} \rceil = 0$. 
\item \label{prop:sigma_smoothCM_2} If $\C_A$ is arithmetically Cohen-Macaulay, then $\sigma = \lceil \frac{c_1+c_2}{d} \rceil \geq \rhp{k[\C_A]}$.
\end{enumerate}
\end{prop} 

\begin{proof}
If $\C_A$ is smooth, then $c_1=c_2=0$ and \ref{prop:sigma_smoothCM_1} follows. 
Now for $s=\lceil \frac{c_1+c_2}{d} \rceil$, the sumset $sA$  decomposes as the union of three disjoint subsets
\[sA = \left( sA \cap C_1\right) \sqcup  \left( sA \cap [c_1,sd-c_2]\right) \sqcup \left( sA \cap \left( sd-C_2 \right)\right)\,.\]
If either $sA \cap C_1\neq C_1$, or $sA \cap [c_1,sd-c_2]\neq [c_1,sd-c_2]$, or $sA \cap \left( sd-C_2 \right)\neq \left( sd-C_2 \right)$, then $E_\sg\neq\emptyset$. Thus, if $\C_A$ is arithmetically Cohen-Macaulay, by applying Proposition \ref{prop:charcM} \ref{prop:charCMe} one gets that 
$sA = C_1 \sqcup  [c_1,sd-c_2] \sqcup \left( sd-C_2 \right)$ and \ref{prop:sigma_smoothCM_2} follows.
\end{proof} 

As a direct consequence of Proposition \ref{prop:sigma_smoothCM} we recover the well-known fact that for any $n\geq 4$, the rational normal curve, i.e., the curve $\C_A$ given by $A = [0,n-1]$, is the only projective monomial curve in $\Pn{n-1}_k$ which is both smooth and arithmetically Cohen-Macaulay.

\begin{ex}\label{ex:sigma}
\begin{enumerate}[(1)]
    \item If $A = [0,d] \setminus \{a\}$ for some $2\leq a\leq d-2$, then $c_1=c_2=0$ and $\sigma = 2$ by Theorem \ref{thm:Granville_Walker}. In this example, $\sigma = \rhp{k[\C_A]} > \lceil \frac{c_1+c_2}{d} \rceil$.
    \item\label{it:ex2} For $A = \{0,2,5,6,9\}$, one has $c_1 = 4$, $c_2=6$ and $\rhp{k[\C_A]} = 1$, so $\sigma=\lceil \frac{c_1+c_2}{d} \rceil = 2 > \rhp{k[\C_A]}$.
\end{enumerate}
\end{ex}

By combining the Erdös-Graham bound for the condutor of a numerical semigroup and the bound for the Castelnuovo-Mumford regularity of a projective monomial curve given by L'vovsky, we obtain the following new bound for the sumsets regularity. This bound is different from the already known bounds recalled in Section \ref{sec:background}. 
Indeed, in Example \ref{ex:sigma} \ref{it:ex2}, both numbers $s_0^{EG}$ and $s_0^L$ introduced in Theorem \ref{thm:boundSigma} are strictly lower than the Granville-Walker bound $s_0^{GW}$ recalled in Section \ref{sec:background}: $s_0^{EG}=4$, $s_0^L=5$, and $s_0^{GW}=6$. This new bound deserves to be studied in the future.

\begin{thm} \label{thm:boundSigma}
If $A=\{a_0=0<a_1<\dots<a_{n-1}=d\}\subset \N$ is a finite set in normal form, set 
\begin{itemize}
\item $s_0^{EG} := \left\lceil 2 \left( \left\lfloor \dfrac{d}{n-1} \right\rfloor (1+\dfrac{a_{n-2}-a_1}{d})-1+\dfrac{1}{d} \right)  \right\rceil$,
%$s_0^{EG} := \left\lceil 2 \left( \dfrac{1}{d} \left\lfloor \dfrac{d}{n-1} \right\rfloor (a_{n-2}+d-a_1)-1+\dfrac{1}{d} \right)  \right\rceil$, 
and 
\item $s_0^L:=\max_{1 \leq i < j \leq n-1} \left\{ (a_i-a_{i-1})+(a_j-a_{j-1}) \right\} -1$.
\end{itemize}
Then, the least integer $\sigma$ such that the decomposition \eqref{eq:struct_thm} in Theorem \ref{thm:structure} holds for all $s\geq\sigma$, i.e., the sumsets regularity of $A$, satisfies
\[\sigma\leq\max \{ s_0^{EG},s_0^L \}\,.\]
\end{thm}

\begin{proof}
By \cite[Thm. 3.1.12]{Alfonsin2005}, one has that $c_1 \leq 2a_{n-2} \lfloor \frac{d}{n-1} \rfloor -d+1$ and $c_2 \leq 2(d-a_1) \lfloor \frac{d}{n-1}\rfloor-d+1$. Combining these two bounds, one gets that $\lceil \dfrac{c_1+c_2}{d} \rceil\leq s_0^{EG}$. On the other hand, using the known fact that $\rhp{k[\C_A]} \leq \reg{k[\C_A]}$ (\cite[Thm 4.2]{Eisenbud_syzygies}) and that $\reg{k[\C_A]}\leq s_0^L$ by 
\cite[Prop. 5.5]{Lvovsky1996}, the upper bound follows from Theorem \ref{thm:structure}.
\end{proof} 

In order to express $\reg{k[\C_A]}$ in terms of $\aps$ and $E_\sg$, let's introduce the following notations.

\begin{defi} \label{def:m1_m2}
For any set $A\subset \N$ in normal form, consider the homogeneous semigroup $\sg\subset\N^2$ associated. We define 
\begin{itemize}
\item $\mE := \max \left( \{ s\in \N: E_{s+1}\neq \emptyset \}\right)$ (and $\mE:=-\infty$ if $E_\sg=\emptyset$), and
\item $\mAP := \max \left( \left\{ s \in \N: \operatorname{AP}_s \neq \emptyset \right\} \right)$.
\end{itemize}  
\end{defi}

\begin{rem} \label{obs:m1m2}
\begin{enumerate}[(1)]
    \item\label{obs:m1m2_1} Note that the maxima in Definition \ref{def:m1_m2} are attained because $\AP_\sg$ and $E_\sg$ are finite subsets of $\N^2$ by Remark \ref{rem:bounds_apsE}. In fact, $\mE \leq \sigma$ and $\mAP\leq \sigma +1$.
    \item\label{obs:m1m2_2} Both $\mE$ and $\mAP$ can be expressed in terms of the sumsets of $A$ as follows:
    \begin{itemize}
        \item $\mE = \max \left( \{s \in \N: \exists \alpha \in sA \text{ such that } \alpha-d\in sA\setminus (s-1)A\} \right)$, and
        \item $\mAP = \max \left( \left\{ s \in \N: \exists \alpha \in sA \text{ such that } \alpha \notin (s-1)A\text{ and }  \alpha-d \notin (s-1)A\right\} \right)$.
    \end{itemize}
\end{enumerate}
\end{rem}

The following result gives a combinatorial way for computing the Castelnuovo-Mumford re\-gu\-la\-ri\-ty of $k[\C_A]$.

\begin{thm} \label{thm:reg}
The Castelnuovo-Mumford regularity of the projective monomial curve $\C_A$ is \[\reg{k[\C_A]} = \max\{ \mE,\mAP \} \, .\]
\end{thm}

In order to prove this result, let's recall some known facts on the local cohomology modules of the coordinate ring of $\C_A$, $k[\C_A]\cong k[\sg]$. For $k[\C_A]$, there are at most two non-trivial local cohomology modules, $\Hom{1}$ and $\Hom{2}$, where $\mathfrak{m}$ denotes the irredundant ideal. Furthermore, these two modules are completely characterized in terms of the semigroup $\sg$.

\begin{lemma}[{\cite[Lemma 2.2]{Hellus2010}}] \label{lemma:Hellus}
Let $G \subset \Z^2$ be the group generated by $\sg$ and $\sg' = G \cap \left( \sg_1 \times \sg_2 \right)$.
\begin{enumerate}[(1)]
    \item \label{lemma:HellusH1}$\Hom{1} \cong k[\sg'\setminus \sg]$, and
    \item \label{lemma:HellusH2}$\Hom{2} \cong k\left[ G \cap \left( (\Z\setminus\sg_1) \times (\Z\setminus\sg_2) \right) \right]$,
\end{enumerate}
where the symbol $\cong$ means that there exists an isomorphism of $\Z$-graded modules.
\end{lemma}

When $\C_A$ is arithmetically Cohen-Macaulay, $\sg'=\sg$ by Proposition \ref{prop:charcM} \ref{prop:charCMf} so $\Hom{1} = 0$ as we already know. For $i=1,2$, let \[\en{\Hom{i}} := \max \{ j: \left(\Hom{i}\right)_j \neq 0 \}\] (with the convention that $\en{\Hom{i}} := -\infty$ if $\Hom{i} = 0$). Then, by the equivalent definition for the the Castelnuovo-Mumford regularity given in \cite{EisenbudGoto1984},  one has that  
\begin{equation} \label{eq:reg_end}
   \reg{k[\sg]} = \max \{ \en{\Hom{1}}+1,\en{\Hom{2}}+2 \} \, .
\end{equation}

The proof of Theorem \ref{thm:reg} will then be a consequence of the following two lemmas that relate the local cohomology modules $\Hom{1}$ and $\Hom{2}$ to the numbers $\mE$ and $\mAP$. Note that the relation 
$\mE = \en{\Hom{1}}+1$ stated in Lemma \ref{lemma:reg_1} also holds when $\C_A$ is arithmetically Cohen-Macaulay since both numbers are $-\infty$ in this case.

\begin{lemma} \label{lemma:reg_1}
If $S'\neq S$, i.e., if $\C_A$ is not arithmetically Cohen-Macaulay, then \[\max\{s: E_{s+2}\neq \emptyset\} = \max\{s: (S'\setminus S) \cap L_s \neq \emptyset \} \, .\]
Therefore, $\mE = \en{\Hom{1}}+1$. 
\end{lemma}

\begin{proof} 
If $\C_A$ is not arithmetically Cohen-Macaulay, then $E_\sg\neq\emptyset$ by Proposition \ref{prop:charcM} \ref{prop:charCMe}.
Set $E_\sg' := \{(x,y) \in \N^2: (x,y)+\bfa_0+\bfa_{n-1} \in E_\sg\}$ and, for each $s\in \N$, $E_s' := E_\sg' \cap L_s$. Note that $(x,y) \in E_s'$ if and only if $(x,y)+\bfa_0+\bfa_{n-1} \in E_{s+2}$ so $\max\{s: E_{s+2} \neq \emptyset\} = \max\{s: E_s' \neq \emptyset\}$.
Let us consider an element $(x,y) \in E_\sg'$. It is clear that $(x,y)\in \sg' \setminus\sg$ since $(x,y) = (x+d,y) -(d,0) \in G$. Therefore, $E_\sg' \subset \sg'\setminus \sg$ and we get that $\max\{ s: E_s'\neq \emptyset \} \leq \max\{s: (S'\setminus S) \cap L_s \neq \emptyset \}$.\\
Conversely, let $(x,y) \in \left( \sg'\setminus \sg \right) \cap L_s$ be an element such that $s$ is maximum. Then, $(x,y)+\bfa_0 \in \sg$ and $(x,y)+\bfa_{n-1} \in \sg$, and hence $(x,y) \in E_s'$. So $\max\{s: E_s'\neq \emptyset\} \geq \max\{s: (S'\setminus S) \cap L_s \neq \emptyset \}$ and the equality $\max\{s: E_{s+2}\neq \emptyset\} = \max\{s: (S'\setminus S) \cap L_s \neq \emptyset \}$ follows.
By Lemma \ref{lemma:Hellus} \ref{lemma:HellusH1}, it implies that $\mE = \en{\Hom{1}}+1$.
\end{proof}

Observe that in the previous proof, we show that $E_\sg' \subset \sg'\setminus \sg$. Equality, which would be a result stronger than the one stated in Lemma \ref{lemma:reg_1}, is wrong in general. Using the example given in \cite[Example 3.2]{Hellus2010}, we show that those two sets may be different.

\begin{ex}\label{ex:Hellus2010}
For $A = \{0,1,2,5,13,14,16,17\}$, the curve $\C_A$ is smooth. Thus, $\sg_1=\sg_2=\N$ and $G = \Z^2$, and hence $\sg' = G\cap \N^2 = \N^2$. Since $(8,9) \in \sg'\setminus \sg$ but $(8,9) \notin E_\sg'$ because $(8,26)\notin\sg$, one has that the inclusion $E_\sg' \subset \sg'\setminus \sg$ is  strict.
\end{ex}

We now want to relate $\mAP$ to $\en{\Hom{2}}$. Let $(x,y) \in G \cap \left( (\Z\setminus\sg_1) \times (\Z\setminus\sg_2) \right) \cap L_s$ be an element with $s$ maximal. Since $x\notin \sg_1$ and $y\notin \sg_2$, one has that $(x,y+d) \notin \sg$ and $(x+d,y) \notin \sg$. There are two possibilities, either $(x+d,y+d)\in \sg$ or $(x+d,y+d) \notin \sg$, and let's check that in both cases, the inequality \eqref{eq:lemma_reg_2} below holds. In the first case, note that $(x+d,y+d) \in \aps\cap L_{s+2}$, so $\max\{s: \AP_{s+2} \neq \emptyset\} \geq \max\{ s: G \cap \left( (\Z\setminus\sg_1) \times (\Z\setminus\sg_2) \right) \cap L_s \neq \emptyset\}$ and \eqref{eq:lemma_reg_2} follows from Lemma \ref{lemma:Hellus} \ref{lemma:HellusH2}. In the second case, there exists an index $i$, $0\leq i\leq d-1$, such that $x\equiv r_i \pmod{d}$ and $y \equiv t_{d-i} \pmod{d}$. Then, $r_i \geq x+d$ and $t_{d-i} \geq y+d$ and since $(x+d,y+d) \notin \sg$, by Lemma \ref{lemma:CMchar} there exist natural numbers $x'\geq x+d$ and $y'\geq y+d$, being at least one of these two inequalities strict, such that $(x',y') \in \aps$. Observe that $(x',y') \in L_{s'}$ for $s'\geq s+3$, so $\max\{s: \AP_{s+2} \neq \emptyset\} > \max\{ s: G \cap \left( (\Z\setminus\sg_1) \times (\Z\setminus\sg_2) \right) \cap L_s \neq \emptyset\}$ in this case. In both cases, one has that
\begin{equation}
\mAP \geq \en{\Hom{2}}+2 \, .
\label{eq:lemma_reg_2}
\end{equation}
Adding an additional hypothesis, one gets equality in \eqref{eq:lemma_reg_2} as the following lemma shows.

\begin{lemma} \label{lemma:reg_2} 
    If $\en{\Hom{2}}+2 > \en{\Hom{1}}+1 = \mE$, then \[\max\{s: \operatorname{AP}_{s+2} \neq \emptyset\} = \max\{ s: G \cap \left( (\Z\setminus\sg_1) \times (\Z\setminus\sg_2) \right) \cap L_s \neq \emptyset\}.\]
    Therefore, in this case one has that $\mAP = \en{\Hom{2}}+2$.
\end{lemma}

\begin{proof}
Let $(x,y) \in \AP_{s+2}$ be an element such that $s$ is maximal and consider the element $(x-d,y-d)$. 
If $(x-d,y-d) \notin G \cap \left( (\Z\setminus\sg_1) \times (\Z\setminus\sg_2) \right)$, one can assume without loss of generality that $x-d\notin\sg_1$. Then, there exists $y'\geq y+d$ such that $(x-d,y') \in \sg$, so $(x,y') \in E_{s'}$ for some $s'\geq s+3$. Therefore, $\en{\Hom{1}}+1 = \mE \geq \mAP$ by Lemma \ref{lemma:reg_1}, and using \eqref{eq:lemma_reg_2} we get that $\en{\Hom{1}}+1\geq \en{\Hom{2}}+2$ which is 
in contradiction with the hypothesis in the statement of the lemma.
Thus, $(x-d,y-d) \in G \cap \left( (\Z\setminus\sg_1) \times (\Z\setminus\sg_2) \right) \cap L_s$, and hence $\en{\Hom{2}} +2 \geq \mAP$ by Lemma \ref{lemma:Hellus} \ref{lemma:HellusH2}. Using \eqref{eq:lemma_reg_2}, we are done. 
\end{proof}

Note that if one removes the hypothesis $\en{\Hom{2}}+2 > \en{\Hom{1}}+1 = \mE$ in Lemma \ref{lemma:reg_2}, the result may be wrong. In order to illustrate this fact, we will use again the example in \cite[Example 3.2]{Hellus2010}.

\begin{ex} 
For $A = \{0,1,2,5,13,14,16,17\}$, as observed in Example \ref{ex:Hellus2010},
$\sg_1=\sg_2=\N$ and $G = \Z^2$. Therefore, $\en{\Hom{2}} = 0$ by Lemma \ref{lemma:Hellus} \ref{lemma:HellusH2}, but $(8,43) \in \operatorname{AP}_3$ so $\mAP \neq \en{\Hom{2}}+2$. 
\end{ex}

\begin{proof}[Proof of Theorem \ref{thm:reg}]
If $\mE \geq \mAP$, then $\mE = \en{\Hom{1}}+1$ by Lemma \ref{lemma:reg_1}, and
$\en{\Hom{1}}+1 \geq \en{\Hom{2}}+2$ because otherwise, by Lemma \ref{lemma:reg_2} one would have that $\mAP>\mE$, a contradiction. Thus, the equality $\reg{k[\C_A]} = \mE$ follows from equation \eqref{eq:reg_end}. \\
Assume now that $\mAP > \mE$ and consider an element $(x,y) \in \AP_s$ with $s=\mAP$. Since $s > \mE$, then
$(x,y)-\bfa_0-\bfa_{n-1} \in G \cap \left( (\Z\setminus \sg_1) \times (\Z\setminus \sg_2)\right) \cap L_{s-2}$ and hence \[\en{\Hom{2}}+2 \geq s=\mAP > \mE = \en{\Hom{1}} \, ,\] 
where the first inequality follows from Lemma \ref{lemma:Hellus} \ref{lemma:HellusH2} and the last equality from Lemma \ref{lemma:reg_1}. Therefore, 
$\en{\Hom{2}}+2>\en{\Hom{1}}+1$ and the equality $\reg{k[\C_A]} = \mAP$ follows from Lemma \ref{lemma:reg_2} and equation \eqref{eq:reg_end}.
\end{proof}

Note that there exist curves such that the maximum in Theorem \ref{thm:reg} is equal to $\mE$ and not equal to $\mAP$, and vice versa. For instance, if $\C_A$ is arithmetically Cohen-Macaulay, then $\mAP > \mE=-\infty$. But there also exist non-arithmetically Cohen-Macaulay curves such that $\mAP>\mE$. In order to illustrate these facts, we use the same curves as in Example \ref{ex:2.13}.
 
\begin{ex}\label{ex:step}
\begin{enumerate}[(1)]\label{ex:stepLast}
    \item For $A = \{0,1,3,11,13\}$, one has that $\mE = 5$ and $\mAP = 4$, so $\C_A$ is not arithmetically Cohen-Macaulay, and $\reg{k[\C_A]} = 5=\mE>\mAP$.
    \item \cite[Ex. 4.3]{Bermejo2006sat}. \label{ex:stepNotLast}
    For $A = \{0,5,9,11,20\}$, $\mE = 4$ and $\mAP=5$, so $\C_A$ is not arithmetically Cohen-Macaulay, and $\reg{k[\C_A]} = 5=\mAP>\mE$.
\end{enumerate}
\end{ex}

\begin{rem}
Let $\mathfrak{m}$ be the maximal homogeneous ideal of $k[\sg] \cong k[\C_A]$ and $\mathfrak{q} := \langle u^d,v^d\rangle$.
We know that $\mathfrak{q}$ is a minimal reduction of $\mathfrak{m}$. Denote by $\red{k[\C_A]}$ the reduction number of $\mathfrak{m}$ with respect to $\mathfrak{q}$, i.e., $\red{k[\C_A]} = \min\{s\in \N: \mathfrak{m}^{s+1}=\mathfrak{q} \mathfrak{m}^s\}$, which can be computed as \[\red{k[\C_A]} = \min\{ s\in \N: (s+1)\bfA = s\bfA + \{\bfa_0,\bfa_{n-1}\} \} \, .\]
By the discussion at the beginning of section \ref{sec:semigroup}, it is clear that $\red{k[\C_A]} = \mAP$, and we can characterize when $\reg{k[\C_A]} = \red{k[\C_A]}$ in a combinatorial way: by Theorem \ref{thm:reg}, \[\reg{k[\C_A]} = \red{k[\C_A]} \Longleftrightarrow \mAP \geq \mE \, .\] In particular, we obtain that $\reg{k[\C_A]} = \red{k[\C_A]}$ whenever $\C_A$ is arithmetically Cohen-Macaulay which is already known. But this occurs in many other examples, e.g., in Example \ref{ex:stepLast} \ref{ex:stepNotLast}.
\end{rem}

The Castelnuovo-Mumford re\-gu\-la\-ri\-ty of the semigroup ring $k[\sg]$ can also be bounded from above and from below in terms of $\sigma$, the conductor of $\sg$. These bounds will be given in Theorem \ref{thm:bounds_reg} where we distinguish two cases depending on the value of $\sigma$ in Theorem \ref{thm:sigma}. Let's first prove a lemma that will be needed in the proof. Recall that for $i=1,2$, the \textit{Fröbenius number} of $\sg_i$, denoted by $F(\sg_i)$, is the largest gap of $\sg_i$, i.e., $F(\sg_i)=c_i-1$.

\begin{lemma}\label{lemma:bounds_reg}
    Set $N := \lceil \frac{c_1+c_2}{d} \rceil$. Then, $\reg{k[\C_A]} \geq \lceil \frac{N}{2} \rceil +1$.
\end{lemma}

\begin{proof} 
    One has that $F(\mathcal{S}_1)+d\in \ap_1$ and consider $y \in \ap_2$ such that $F(\mathcal{S}_1)+d+y \equiv 0 \pmod{d}$. Note that $y\neq 0$. By Lemma \ref{lemma:CMchar}, there are two options: either $(F(\mathcal{S}_1)+d,y) \in \aps$, or there exist $y'\geq y$ such that $(F(\mathcal{S}_1)+d,y') \in \aps$. In both cases, there exists $y\geq 1$ such that $(F(\mathcal{S}_1)+d,y)\in \aps$ and, analogously, there exists $x\geq 1$ such that $(x,F(\mathcal{S}_2)+d) \in \aps$. By Theorem \ref{thm:reg}, 
    \[\reg{k[\C_A]} \geq \max \left\{ \dfrac{F(\mathcal{S}_1)+d+y}{d}, \dfrac{F(\mathcal{S}_2)+d+x}{d} \right\} \geq \dfrac{1}{2} \dfrac{F(\mathcal{S}_1)+F(\mathcal{S}_2)+2}{d}+1 = \dfrac{c_1+c_2}{2d}+1.\] Thus, $\reg{k[\C_A]} \geq \lceil \frac{c_1+c_2}{2d} \rceil +1=\lceil \frac{N}{2} \rceil+1$.
\end{proof}

\begin{thm} \label{thm:bounds_reg}
We have the following bounds for the Castelnuovo-Mumford regularity of $k[\C_A]$:
    \begin{enumerate}[(1)]
        \item If $\sigma = \rhp{k[\C_A]} \geq \lceil \frac{c_1+c_2}{d} \rceil$, then $\sigma \leq \reg{k[\C_A]} \leq \sigma +1$.
        \item\label{thm:bounds_reg_2} If $\sigma = \lceil \frac{c_1+c_2}{d} \rceil > \rhp{k[\C_A]}$, then $\lceil \frac{\sigma}{2} \rceil +1 \leq \reg{k[\C_A]} \leq \sigma +1$.
    \end{enumerate}
\end{thm}

\begin{proof}
    In both cases, the upper bound is a consequence of Theorem \ref{thm:reg} and Remark \ref{obs:m1m2} \ref{obs:m1m2_1}.
    If $\sigma = \rhp{k[\C_A]} \geq \lceil \frac{c_1+c_2}{d} \rceil$, then we apply the known fact $\rhp{k[\C_A]} \leq \reg{k[\C_A]}$, see \cite[Thm 4.2]{Eisenbud_syzygies}, and in the other case, the lower bound is the one given in Lemma \ref{lemma:bounds_reg}.
\end{proof}

\begin{ex} \label{ex:sharp}
In order to illustrate that all the upper and lower bounds in Theorem \ref{thm:bounds_reg} are sharp, the values of $\rhp{k[\C_A]}$, $\lceil \frac{c_1+c_2}{d} \rceil$, $\sigma$ and $\reg{k[\C_A]}$ in four different examples are displayed in Table \ref{tab:bounds_sharp}.
\end{ex}
\begin{table}[htbp]
    \centering
    \caption{Examples where the bounds in Theorem \ref{thm:bounds_reg} are attainted.}
    \label{tab:bounds_sharp}
    \begin{tabular}{lcccc}
    \toprule
    $A$ & $\rhp{k[\C_A]}$ & $\lceil \frac{c_1+c_2}{d} \rceil$ & $\sigma$ & $\reg{k[\C_A]}$ \\
    \midrule
     $\{0,1,3,11,13\}$ & $5$ & $1$ & $5$ & $5$\\
     $\{0,1,3,5,6,12\}$ & $1$ & $1$ & $1$ & $2$\\
     $\{0,4,5,9,16\}$ & $2$ & $3$ & $3$ & $3$\\
     $\{0,5,9,11,20\}$ & $3$ & $4$ & $4$ & $5$\\
     \bottomrule
    \end{tabular}
    \label{tab:ex_bounds}
\end{table}

The following result is more precise than the one stated in Theorem \ref{thm:bounds_reg} in a particular case. It gives, in this case, the precise relationship between the three regularities, in the sense of Castelnuovo-Mumford, of the Hilbert function and of the sumsets. 

\begin{prop}\label{prop:CMspecial}
If $\C_A$ is arithmetically Cohen-Macaulay and $\left(F(\mathcal{S}_1)+d,F(\mathcal{S}_2)+d\right) \in \aps$, then \[\sigma = \left\lceil \frac{c_1+c_2}{d} \right\rceil,\ 
\rhp{k[\C_A]} = \sigma,\text{ and }\reg{k[\C_A]} = \sigma+1\,.\]
\end{prop}

\begin{proof} 
Since $\left(F(\mathcal{S}_1)+d,F(\mathcal{S}_2)+d\right)\in\aps$, one has that 
$\left(F(\mathcal{S}_1)+d,F(\mathcal{S}_2)+d\right)\in\operatorname{AP}_s$ for $s=\mAP$ and, as $\C_A$ is arithmetically Cohen-Macaulay, $\reg{k[\C_A]} =\mAP$ by Theorem \ref{thm:reg}. Thus, \[\reg{k[\C_A]} = \dfrac{F(\mathcal{S}_1)+d+F(\mathcal{S}_2)+d}{d} = \dfrac{F(\mathcal{S}_1)+F(\mathcal{S}_2)}{d}+2 \, .\] 
On the other hand, \[\left\lceil \dfrac{c_1+c_2}{d} \right\rceil = \left\lceil \dfrac{F(\mathcal{S}_1)+F(\mathcal{S}_2)}{d}+\dfrac{2}{d} \right\rceil = \dfrac{F(\mathcal{S}_1)+F(\mathcal{S}_2)}{d}+1 \, ,\] so $\reg{k[\C_A]} = \lceil \frac{c_1+c_2}{d} \rceil +1$, and $\rhp{k[\C_A]}=\lceil \frac{c_1+c_2}{d} \rceil$ since $\reg{k[\C_A]}=\rhp{k[\C_A]}+1$ whenever $\C_A$ is arithmetically Cohen-Macaulay. Finally, $\sigma=\lceil \frac{c_1+c_2}{d} \rceil$ by Theorem \ref{thm:sigma}.
\end{proof}

\begin{ex}
For $A = \{0,1,2,3,8\}$, $k[\C_A]$ is Cohen-Macaulay as shown in Example \ref{ex:01238}, and $(F(\sg_1)+d,F(\sg_2)+d) = (7,17) \in \aps$. By Proposition \ref{prop:CMspecial}, $\sigma = \rhp{k[\C_A]} = \lceil \frac{c_1+c_2}{d} \rceil = 3$, and $\reg{k[\C_A]} = \sigma+1 = 4$.
\end{ex}

Using the previous results, we can give a new proof for the bound obtained by J. Elias in \cite{Elias2022} for arithmetically Cohen-Macaulay curves.

\begin{prop}[{\cite[Thm. 4.7]{Elias2022}}]
    If $A = \{a_0=0<a_1<\dots<a_{n-1}=d\} \subset \N$ is a set in normal form such that $\C_A$ is arithmetically Cohen-Macaulay, then \[\reg{k[\C_A]} \leq \left\lceil \dfrac{d-1}{n-2} \right\rceil.\]
\end{prop}

\begin{proof}
Set $s_0:=\lceil \frac{d-1}{n-2}\rceil$. By Corollary \ref{cor:CMseq}, the sequence $\left( |sA|-|(s-1)A|\right)_{s\in \N}$ is increasing and its limit is $d$. 
Indeed, as observed in the proof of this corollary, $|sA|-|(s-1)A|=\sum_{j=0}^s |\AP_j|$ for all $s\in \N$. On the other hand,  $|\aps|=d$ by Proposition \ref{prop:charcM} \ref{prop:charCMd} and, by \cite[Thm. 1]{Lev1996}, $|sA|-|(s-1)A|\geq d$ if $s\geq s_0$. Therefore, $|\AP_s|=0$ for all $s>s_0$, and hence $\reg{k[\C_A]}\leq s_0$ by Theorem \ref{thm:reg}.
\end{proof}

Finally, as a consequence of Theorem \ref{thm:bounds_reg}, one gets a sufficient condition for $\sigma$ to be equal to $\lceil \frac{c_1+c_2}{d} \rceil$ in Theorem \ref{thm:sigma}. The condition is expressed in terms of the difference between the Castelnuovo-Mumford regularity and the regularity of the Hilbert function of $k[\C_A]$. We will see in the next section how this condition can be characterized in terms of the Betti numbers of $k[\C_A]$.

\begin{cor}\label{cor:sigma_diff}
If $D = \reg{k[\C_A]}-\rhp{k[\C_A]} \geq 2$, then $\sigma = \lceil \frac{c_1+c_2}{d} \rceil > \rhp{k[\C_A]}$.
\end{cor}

\begin{proof}
If $\sigma = \rhp{k[\C_A]} \geq \lceil \frac{c_1+c_2}{d} \rceil$, then $\sigma \leq \reg{k[\C_A]} \leq \sigma +1$ by Theorem \ref{thm:bounds_reg}, so $D\leq 1$.
\end{proof}

\section{The shape of the Betti diagram} \label{sec:shape}
In this section, we relate both regularities $\reg{k[\C_A]}$ and $\rhp{k[\C_A]}$ in terms of the Betti diagram of $k[\C_A]$ that can be used to characterize the difference $D:=\reg{k[\C_A]}-\rhp{k[\C_A]}$. Recall that the projective dimension of $k[\C_A]$ is either $n-2$ if the ring $k[\C_A]$ is Cohen-Macaulay, or $n-1$ otherwise. \newline

The Hilbert function and polynomial of $k[\C_A]$ are computed using the Betti numbers as follows: if we denote by $\regn := \reg{k[\C_A]}$ the Castelnuovo-Mumford regularity of $k[\C_A]$, then
\begin{equation*}
\begin{split}
    \hf{A}(t) &= \binom{t+n-1}{n-1} + \sum_{i=1}^{n-1} \sum_{j=0}^{\regn} (-1)^i \beta_{i,i+j} \binom{t-(i+j)+n-1}{n-1} \, , \\
    \hp{A}(t) &= \dfrac{1}{(n-1)!} \left[ (t+n-1)(t+n-2)\dots (t+1) + \sum_{i=1}^{n-1} \sum_{j=0}^{\regn} (-1)^i \beta_{i,i+j} \prod_{l=1}^{n-1} \left( t-(i+j)+l \right) \right].
\end{split}
\end{equation*}
Taking into account the roots of the polynomial $\prod_{l=1}^{n-1} \left( t-(i+j)+l\right)$, it is easy to prove that $\hf{A}(t)=\hp{A}(t)$ for all $t\geq \regn$, i.e. 
\begin{equation}\label{eq:rleqreg}
\rhp{k[\C_A]} \leq \reg{k[\C_A]};
\end{equation}
see \cite[Thm. 4.2]{Eisenbud_syzygies} for the details. In order to determine precisely the difference $D$ between the two regularities, we need to evaluate the difference $\hp{A}(\regn-\lambda)-\hf{A}(\regn-\lambda)$ for $1\leq\lambda \leq \regn$. For $\lambda\geq 1$, set
\[A_{i+j}^{(\lambda)} := \binom{\regn+n-(\lambda+1)-(i+j)}{n-1}, \qquad  B_{i+j}^{(\lambda)} := \dfrac{1}{(n-1)!} \prod_{l=1}^{n-1} \left( \regn-\lambda-(i+j)+l\right).\] 
Using this notation, for all $\lambda$, $1\leq \lambda \leq \regn$, we can write 
\begin{equation}\label{eq:dif_hfhp}
\begin{split}
    \hp{A}(\regn-\lambda)-\hf{A}(\regn-\lambda) &= \sum_{i=1}^{n-1} \sum_{j=0}^{\regn} (-1)^i \beta_{i,i+j} \left( B_{i+j}^{(\lambda)} - A_{i+j}^{(\lambda)} \right) \\
    &= \sum_{i+j=1}^{\regn+n-1}  (-1)^i \beta_{i,i+j} \left( B_{i+j}^{(\lambda)} - A_{i+j}^{(\lambda)} \right).
\end{split}
\end{equation}
The following lemma establishes when $A_{i+j}^{(\lambda)}$ and $B_{i+j}^{(\lambda)}$ coincide.
\begin{lemma}\label{lemma:AB}
    Consider $\lambda\geq 1$ and $i+j$ such that $1\leq i+j\leq \regn+n-\lambda$.
    \begin{enumerate}[(1)]
        \item\label{lemma:ABa} If $i+j\leq \regn-\lambda$, then $A_{i+j}^{(\lambda)} = B_{i+j}^{(\lambda)} \neq 0$.
        \item\label{lemma:ABb} If $\regn-\lambda+1 \leq i+j \leq \regn+n-(\lambda+1)$, then $A_{i+j}^{(\lambda)} = B_{i+j}^{(\lambda)} = 0$.
        \item\label{lemma:ABc} If $i+j = \regn+n-\lambda$, then $A_{i+j}^{(\lambda)} = 0$ and $B_{i+j}^{(\lambda)} = (-1)^{n-1}$.
    \end{enumerate}
\end{lemma}

\begin{proof}
If $i+j\leq \regn-\lambda$, then $\regn+n-(\lambda+1)-(i+j) \geq n-1$ so $A_{i+j}^{(\lambda)} = B_{i+j}^{(\lambda)} \neq 0$ and $\ref{lemma:ABa}$ follows.
Otherwise, $A_{i+j}^{(\lambda)} = 0$ and we distinguish two cases.
If $i+j \leq \regn+n-(\lambda+1)$, then $1\leq i+j+\lambda -\regn \leq n-1$ and hence $B_{i+j}^{(\lambda)} = 0$ and \ref{lemma:ABb} follows. 
Finally, if $i+j = \regn+n-\lambda$, then \[B_{\regn+n-\lambda}^{(\lambda)} = \dfrac{1}{(n-1)!} \prod_{l=1}^{n-1} (l-n) = (-1)^{n-1},\] and we are done.
\end{proof}

By Lemma \ref{lemma:AB} \ref{lemma:ABc} and equation \eqref{eq:dif_hfhp}, $\hp{A}(\regn-1) -\hf{A}(\regn-1) = \beta_{n-1,\regn+n-1}$ so if $\beta_{n-1,\regn+n-1} \neq 0$, one gets that $\rhp{k[\C_A]} = \reg{k[\C_A]}$, i.e., $D=0$. And the reciprocal statement also holds. This is a particular case of the following result that relates precisely $D$ to some of the Betti numbers.

\begin{prop}\label{prop:BettiReg}
If $\lambda\geq 1$ be the least positive integer such that $\sum_{i} (-1)^i \beta_{i,\regn+n-\lambda} \neq 0$, then
$\rhp{k[\C_A]} = \reg{k[\C_A]}-\lambda+1$, i.e., $D = \lambda-1$, and
\[\hp{A}(\regn-\lambda)-\hf{A}(\regn-\lambda) = \sum_{i} (-1)^{n+i-1} \beta_{i,\regn+n-\lambda} \, .\]
\end{prop}

\begin{proof}
The case $\lambda=1$ is proved just before the proposition so assume that $\lambda\geq 2$. Since for all $\mu=1,2,\ldots,\lambda-1$,
$\sum_{i} (-1)^i \beta_{i,\regn+n-\mu} = 0$, by equation \eqref{eq:dif_hfhp} one gets that $\hp{A}(t) = \hf{A}(t)$ for all $t\geq \regn-\lambda+1$, i.e., $\rhp{k[\C_A]} \leq \reg{k[\C_A]}-\lambda+1$. Moreover, by applying Lemma \ref{lemma:AB} \ref{lemma:ABc} to equation \eqref{eq:dif_hfhp}, we obtain that \[\hp{A}(\regn-\lambda)-\hf{A}(\regn-\lambda) = \sum_i (-1)^i \beta_{i,\regn+n-\lambda} B_{\regn+n-\lambda}^{(\lambda)} = \sum_i (-1)^{n+i-1} \beta_{i,\regn+n-\lambda} \neq 0 \, ,\] and hence $\rhp{k[\C_A]} = \reg{k[\C_A]}-\lambda+1$.
\end{proof}

Note that the previous result is general. In its proof, we do not use that the ring is the coordinate ring of a monomial projective curve. This proves the following result that improves \cite[Thm. 4.2]{Eisenbud_syzygies}:

\begin{thm}\label{thm:shapeGeneral}
Let $M$ be a finitely generated graded module over $k[x_0,\ldots,x_{n-1}]$, and denote by $D$ the difference between the Castelnuovo-Mumford regularity and the regularity of the Hilbert function of $M$, i.e., $D:=\reg{M}-\rhp{M}$. Then, $D+1$ is the least non-negative integer $\lambda\geq 0$ such that $\sum_{i} (-1)^i \beta_{i,\reg{M}+n-\lambda} \neq 0$, where the $\beta_{ij}$ are the graded Betti numbers of $M$.  
\end{thm}

\begin{rem}\label{rem:Dgeneral}
    \begin{enumerate}[(1)]
        \item If we focus on the secondary diagonals of the Betti diagram starting from the bottom right of the table, the number $\lambda$ in the previous theorem is the label of the first diagonal such that the alternating sum of the Betti numbers on this diagonal is not $0$.
%\suspend{enumerate}
{\centering

\begin{tikzpicture}
    \matrix [%
      every node/.style={scale=0.8},
      matrix of math nodes,
      nodes={minimum size=5mm},
      column sep=1em,
      row sep=0.5em,
      anchor=west
    ] (Betti) {%
      j/i & { } & 0 & 1 &[-1.5em] \dots &[-1.5em] p-1 & p & { } & p+1 & \dots & n & { }\\%[-1em]
      0 & { } & 1 & \beta_{1,1} & \dots & \beta_{p-1,p-1} & \beta_{p,p} & { } & 0 & \dots & 0 & { }\\
      1 & { } & - & \beta_{1,2} & \dots & \beta_{p-1,p} & \beta_{p,p+1} & { } & 0 & \dots & 0 & { }\\
      \vdots & { } & \vdots & \vdots & \ddots & \vdots & \vdots & { } & \vdots & \ddots & \vdots & { }\\
      \regn-1 & { } & - & \beta_{1,\regn} & \dots & \beta_{p-1,\regn+p-2} & \beta_{p,\regn+p-1} & { } & 0 & \dots & 0 & { }\\
      \regn & { } & - & \beta_{1,\regn+1} & \dots & \beta_{p-1,\regn+p-1} & \beta_{p,\regn+p} & { } & 0 & \dots & 0 & { }\\
      { } & { } & { } & { } & { } & { } & { } & { } & { } & { } & { } & { }\\[0.5em,between origins]
      { } & { } & { } & { } & {\color{blue}\lambda = n-p+1} & \hspace{2mm} {\color{red}\lambda=n-p} & { } & { } & { } & {\color{ForestGreen}\lambda=0} & { } & { }\\
    };
    \path
          % Diagonal lambda=n-p+1
          (Betti-7-5)                          edge[blue]    (Betti-6-6)
          (Betti-6-6)                          edge[blue]    (Betti-5-7)
          (Betti-5-7)                          edge[->,blue] (Betti-4-8)
          % Diagonal lambda=n-p
          (Betti-7-6)                          edge[red]    (Betti-6-7)
          (Betti-6-7)                          edge[->,red] (Betti-5-8)
          % Diagonal lambda=0
          (Betti-7-10)                          edge[ForestGreen]    (Betti-6-11)
          (Betti-6-11)                          edge[->,ForestGreen] (Betti-5-12);
    \draw[black] (Betti-1-1.south west) -- (Betti-1-12.south west);
    \draw[black] (Betti-1-2.north west) -- (Betti-6-2.south west);
    \draw[black] (Betti-1-8.north east) -- (Betti-6-8.south east);
  \end{tikzpicture}
  }
%\resume{enumerate}   
        \item\label{rem:Dgeneral2} If $p$ denotes the projective dimension of the module $M$, the previous result implies that $\beta_{p,\reg{M}+p}\neq 0$, i.e., the regularity is attained at the last step of a m.g.f.r. of $M$, if, and only if, $\lambda=n-p$, i.e., $D=n-1-p$. This occurs, in particular, whenever $M$ is a Cohen-Macaulay module so, in this case, $\reg{M} - \rhp{M}=n-1-p$ which is a well-known fact; see, e.g., \cite[Cor. 4.8]{Eisenbud_syzygies}.
    \end{enumerate}
\end{rem}

 Going back to the case of projective monomial curves, the case $\lambda =2$ (i.e., $D=1$) in Proposition \ref{prop:BettiReg} includes all arithmetically Cohen-Macaulay curves since $\beta_{n-1,\regn+n-1} = \beta_{n-1,\regn+n-2} =0$ and $\beta_{n-2,\regn+n-2} \neq 0$ in this case. But there are also non-arithmetically Cohen-Macaulay curves $\C_A$ such that $D = \reg{k[\C_A]}-\rhp{k[\C_A]} = 1$ as we will see in the next example.

%\enlargethispage{10mm}
\begin{ex}
Different values of $D = \reg{k[\C_A]}-\rhp{k[\C_A]}$ and different shapes for the Betti diagram of $k[\C_A]$ are obtained in the following four examples of monomial curves in $\Pn{4}_k$.
\vspace{-1mm}
\begin{multicols}{2}
\begin{enumerate}
\item[(1)] For $A = \{0,1,3,11,13\}$, $D = 0$ and $\reg{k[\C_A]}$ is attained at the last step of a m.g.f.r.
    \begin{Verbatim}[fontsize=\small]
           0     1     2     3     4
------------------------------------
    0:     1     -     -     -     -
    1:     -     1     -     -     -
    2:     -     2     2     -     -
    3:     -     2     2     -     -
    4:     -     3     8     5     -
    5:     -     -     2     4     2
------------------------------------
total:     1     8    14     9     2
    \end{Verbatim}

\vspace{5mm}

\item[(3)] For $A = \{0,6,9,13,22\}$, $D=1$ and $\reg{k[\C_A]}$ is not attained at the last step of a m.g.f.r.
    \begin{Verbatim}[fontsize=\small]
           0     1     2     3     4
------------------------------------
    0:     1     -     -     -     -
    1:     -     1     -     -     -
    2:     -     1     -     -     -
    3:     -     -     1     -     -
    4:     -     5     9     5     1
    5:     -     -     2     2     -
------------------------------------
total:     1     7    12     7     1
    \end{Verbatim}
\columnbreak
\item[(2)] For $A = \{0,2,5,6,9\}$, $D = 1$ and $\reg{k[\C_A]}$ is attained at the last step of a m.g.f.r.
    \begin{Verbatim}[fontsize=\small]
           0     1     2     3
------------------------------
    0:     1     -     -     -
    1:     -     1     -     -
    2:     -     7    12     5
------------------------------
total:     1     8    12     5
    \end{Verbatim}
\vspace{17.6mm}

\item[(4)] For $A = \{0,5,9,11,20\}$, $D = 2$ and $\reg{k[\C_A]}$ is not attained at the last step of a m.g.f.r.
\begin{Verbatim}[fontsize=\small]
           0     1     2     3     4
------------------------------------
    0:     1     -     -     -     -
    1:     -     1     -     -     -
    2:     -     1     -     -     -
    3:     -     1     1     -     -
    4:     -     3     9     5     1
    5:     -     -     -     1     -
------------------------------------
total:     1     6    10     6     1
    \end{Verbatim}
\end{enumerate}
\end{multicols}
\end{ex}

Recall that, as stated in \cite[Thm. 3.11]{Schenzel1998}, the regularity is always determined by the tail of a m.g.f.r.. In other words, the definition of regularity given in \eqref{eq:defReg} can be simplified as
\[\reg{k[\C_A]} := \max \{j-i: \beta_{i,j} \neq 0,\ n-\dim(k[\C_A])\leq i\leq n-\depth{k[\C_A]},\ j\geq 0\}.\]
In our situation, $\dim(k[\C_A])=2$ and hence $1\leq\depth{k[\C_A]}\leq 2$. If $k[\C_A]$ is Cohen-Macaulay, then the regularity is always attained at the last step of a m.g.f.r., a general and well-known fact. When $k[\C_A]$ is not Cohen-Macaulay, it is always attained at one of the last two steps of a m.g.f.r. and our next result characterizes when the regularity is attained at the last step in terms of the formula given in Theorem \ref{thm:reg} and of the difference $D=\reg{k[\C_A]}-\rhp{k[\C_A]}$.
\begin{thm} \label{thm:last_step}
If $\C_A$ is not arithmetically Cohen-Macaulay, the following are e\-qui\-va\-lent:
\begin{enumerate}[(1)]
    \item\label{thm:last_step1} The Castelnuovo-Mumford regularity of $k[\C_A]$ is attained at the last step of a m.g.f.r.
    \item\label{thm:last_step2} $\reg{k[\C_A]} = \mE$, i.e., $\mE\geq\mAP$.
    \item\label{thm:last_step3} $\reg{k[\C_A]}=\rhp{k[\C_A]}$, i.e., $D=0$.
\end{enumerate}
\end{thm}

\begin{proof}
The equivalence \ref{thm:last_step1}$\Leftrightarrow$\ref{thm:last_step3} is a direct consequence of Proposition \ref{prop:BettiReg} as observed in Remark \ref{rem:Dgeneral} \ref{rem:Dgeneral2}. Therefore, we only have to prove \ref{thm:last_step1}$\Leftrightarrow$\ref{thm:last_step2}. It is well known that the maximal degree of the minimal $(n-\depth{R/I})$-syzygies of $k[\C_A]$ is equal to $\en{\Hom{\depth{R/I}}}+n$. This is, e.g., a consequence of  \cite[Cor. 2.2]{Chardin2007}.
If $k[\C_A]$ is not Cohen-Macaulay, then by Theorem \ref{thm:reg}, its proof, and Lemma \ref{lemma:reg_1}, one has that $\reg{k[\C_A]}=\mE$ if and only if $\en{\Hom{1}}+1 = \mE\geq\en{\Hom{2}}+2$, i.e., if and only if the Castelnuovo-Mumford regularity is attained at the last step of a m.g.f.r. of $k[\C_A]$ by \eqref{eq:reg_end} and the previous observation. This proves the equivalence between \ref{thm:last_step1} and \ref{thm:last_step2}.
\end{proof}

Finally, let's focus on monomial curves in $\Pn{3}_k$. Since these curves have codimension 2, they have some additional properties.

\begin{prop}\label{prop:curvesP3}
Let $A\subset \N$ be a set in normal form with $|A|=4$ and consider the associated monomial curve $\C_A\subset \Pn{3}_k$. 
\begin{enumerate}[(1)]
    \item\label{prop:curvesP3_1} The Castelnuovo-Mumford regularity is attained at the last step of a m.g.f.r. of $k[\C_A]$.
    \item\label{prop:curvesP3_2} 
    Setting $D := \reg{k[\C_A]} - \rhp{k[\C_A]}$, one has that $0\leq D \leq 1$. More precisely, 
    \[\begin{split}
    D=0 &\Longleftrightarrow k[\C_A] \text{ is not Cohen-Macaulay} \Longleftrightarrow \reg{k[\C_A]} = \mE\geq \mAP, \\
    D=1 &\Longleftrightarrow k[\C_A] \text{ is Cohen-Macaulay} \Longleftrightarrow \reg{k[\C_A]} = \mAP > \mE.
    \end{split}\]
\end{enumerate}
\end{prop}

\begin{proof}
\ref{prop:curvesP3_1} is a particular case of \cite[Cor. 2.13]{Bermejo2006cas}. By Proposition \ref{prop:BettiReg} and Remark \ref{rem:Dgeneral} \ref{rem:Dgeneral2}, this implies that either $D=0$ if $\C_A$ is not arithmetically Cohen-Macaulay, or $D = 1$ if $\C_A$ is arithmetically Cohen Macaulay. \ref{prop:curvesP3_2} then follows from Theorem \ref{thm:last_step}.
\end{proof}

\section{The relation between known bounds for $\sigma$ and $\reg{k[\C_A]}$} \label{sec:consequences}

In this final section, we show how the bound for $\sigma$ recently obtained by Granville and Walker in \cite{Granville2021} and the classical bound for $\reg{k[\C_A]}$ given by the Gruson-Lazarsfeld-Peskine Theorem \cite{Gruson1983} are related. As a consequence of some of our results, we obtain that each of these bounds can be deduced from the other. \newline

Let's first recall these two bounds. Consider 
$A = \{a_0=0<a_1<\dots<a_{n-1}=d\} \subset \N$ a set in normal form.
In section \ref{sec:background}, we presented several upper bounds for $\sigma$, the sumsets regularity of $A$, which we also called the conductor of the homogeneous semigroup $\sg$. The best bound is the one given in \cite[Thm. 1]{Granville2021} by Granville and Walker:
\begin{equation}\label{eq:sigmaBound}\sigma \leq          d-n+2\,.\end{equation}

On the other hand, a classical and important result in algebraic geometry provides an upper bound for the Castelnuovo-Mumford regularity of any reduced and irreducible projective curve in terms of its degree and codimension. It is the Gruson-Lazarsfeld-Peskine Theorem; see, e.g., \cite[Thm. 5.1]{Eisenbud_syzygies}. Applied to the monomial projective curve $\C_A$, the Gruson-Lazarsfeld-Peskine Theorem claims that
\begin{equation}\label{eq:regBound}\reg{k[\C_A]}\leq d-n+2\,.\end{equation}

Let's first show that the Granville-Walker bound \eqref{eq:sigmaBound} can be deduced  from \eqref{eq:regBound} using Theorem \ref{thm:sigma}. We start with the following result that bounds one of the terms in Theorem \ref{thm:sigma}.

\begin{lemma} \label{lemma:bound}
    Let $A = \{a_0=0<a_1<\dots<a_{n-1} = d\}\subset \N$ be a set in normal form. If $\sg_1$ is the semigroup generated by $A$, $\sg_2$ is the semigroup generated by $d-A$, and $c_i$ is the conductor of $\sg_i$ for $i=1,2$, then \[\left\lceil \dfrac{c_1+c_2}{d} \right\rceil \leq d-n+1 \, .\]
\end{lemma}

\begin{proof}
The conductors of both semigroups $\sg_1$ and $\sg_2$ can be bounded using \cite[Thm. 3.1.1]{Alfonsin2005}: \[\begin{split}
c_1 &\leq (a_1-1)(a_{n-1}-1)=(a_1-1)(d-1) \, , \\
c_2 &\leq (d-a_{n-2}-1)(d-a_1-1) \, .
\end{split}\]
Therefore, 
\[\begin{split}
c_1+c_2 &\leq d^2-3d-da_{n-2}+a_1a_{n-2}+a_{n-2}+2\\
&\leq d^2-3d-d(a_1+n-3)+(d-1)a_1+(d-1)+2
\end{split}\]
because $a_{n-2} \geq a_1+n-3$ and $a_{n-2}\leq d-1$. Thus,
\[c_1+c_2 \leq d^2-3d-nd+3d-a_1+d+1 \leq d^2-nd+d = d(d-n+1)\,.\]
Dividing by $d$, the result follows.
\end{proof}

As recalled in \eqref{eq:rleqreg}, $\rhp{k[\C_A]} \leq \reg{k[\C_A]}$, so \eqref{eq:sigmaBound} is a straightforward consequence of \eqref{eq:regBound}, Lemma \ref{lemma:bound}, and Theorem \ref{thm:sigma}. \newline

Conversely, in order to show that \eqref{eq:regBound} can be deduced from \eqref{eq:sigmaBound}, we will use the additional result of Granville and Walker  recalled in Theorem \ref{thm:Granville_Walker} where all the sets $A$ in normal form  such that the bound in \eqref{eq:sigmaBound} is attained are characterized. 
We distinguish three cases:
\begin{enumerate}[(a)]
    \item If neither $A$ nor $d-A$ belongs to the two families listed in Theorem \ref{thm:Granville_Walker}, then $\sigma\leq d-n+1$, and \eqref{eq:regBound} follows from Theorem \ref{thm:bounds_reg}.
    \item If $A = [0,d] \setminus \{a\}$ for some $a\in [2,d-2]$, then $\sigma = 2$ and $\reg{k[\C_A]}=2$ as well by Theorem \ref{thm:reg}, and hence \eqref{eq:regBound} holds for such a set $A$. Observe that in this case, $d= n$ so equality holds in  \eqref{eq:sigmaBound} and \eqref{eq:regBound}. 
    \item If $A = [0,1] \sqcup [a+1,d]$ for some $a\in [2,d-2]$, then $sA = [0,sd]$ for all $s\geq a$ and $a\notin (a-1)A$. Therefore $\sigma=a$ and $\reg{k[\C_A]} = a$ by Theorem \ref{thm:reg}, so \eqref{eq:regBound} also follows from \eqref{eq:sigmaBound} in this case. One gets the same conclusion if $d-A=[0,1] \sqcup [a+1,d]$ for some $a \in [2,d-2]$.
\end{enumerate}
Note that this discussion provides a new combinatorial proof of the Gruson-Lazarsfeld-Peskine Theorem for projective monomial curves. \newline

\section*{Acknowledgements}
We want to thank Ignacio García-Marco for suggesting that the result in Theorem \ref{thm:boundSigma} might hold, and the anonymous referee for her/his careful reading of our paper and the interesting comments included in her/his report.

\end{document}